\documentclass[letterpaper,  12pt]{amsart}

\usepackage{amsmath} 

\usepackage{amssymb} 

\usepackage{amsthm}

\usepackage{amsfonts}

\usepackage{mathpazo}
\usepackage[euler-digits, euler-hat-accent]{eulervm}

\usepackage[T1]{fontenc}

\usepackage{mathtools} \usepackage{enumerate}
\usepackage[colorlinks=true, linkcolor=blue, citecolor=blue,
pagebackref=true]{hyperref}

\usepackage{mathtools} \usepackage{etoolbox}
\patchcmd{\section}{\scshape}{\bfseries}{}{} \makeatletter
\renewcommand{\@secnumfont}{\bfseries} \makeatother

\makeatletter \renewenvironment{proof}[1][\proofname]
{\par\pushQED{\qed}\normalfont\topsep6\p@\@plus6\p@\relax\trivlist\item[\hskip\labelsep\bfseries#1\@addpunct{.}]\ignorespaces}{\popQED\endtrivlist\@endpefalse}
\makeatother

\renewcommand{\leq}{\leqslant} \renewcommand{\geq}{\geqslant}

\theoremstyle{plain} \newtheorem{theorem}{Theorem}[section]

\theoremstyle{plain} \newtheorem*{theorem*}{Theorem}

\theoremstyle{plain} 

\theoremstyle{plain} \newtheorem*{example*}{Example}

\theoremstyle{definition} \newtheorem*{question*}{Question}

\theoremstyle{definition} 

\theoremstyle{definition} \newtheorem{remark}[theorem]{Remark}

\theoremstyle{plain} \newtheorem{lemma}[theorem]{Lemma}

\theoremstyle{plain} 

\theoremstyle{plain} \newtheorem*{conjecture*}{Conjecture}

\theoremstyle{plain} \newtheorem{proposition}[theorem]{Proposition}

\theoremstyle{plain} \newtheorem{corollary}[theorem]{Corollary}

\theoremstyle{definition} \newtheorem*{remark*}{Remark}

\newcommand{\bad}{\mathbf{BA}}

\newcommand{\well}{\mathbf{WA}}

\newcommand{\bone}{\mathbb{1}}

\newcommand{\DD}{\mathbb{D}} 

\newcommand{\NN}{\mathbb{N}}

\newcommand{\QQ}{\mathbb{Q}} 

\newcommand{\RR}{\mathbb{R}}

\newcommand{\WW}{\mathbb{W}}

\newcommand{\ZZ}{\mathbb{Z}}

\newcommand{\calD}{\mathcal D}

\newcommand{\bk}{\mathbf k}

\newcommand{\fixx}{\Phi}

\newcommand{\fixy}{\Theta} 

\newcommand{\recips}{\mathcal D'}

\newcommand{\simplex}{\mathcal R}

\newcommand{\comp}{^{\mathsf{c}}}

\DeclarePairedDelimiter{\norm}{\lVert}{\rVert}
\DeclarePairedDelimiter{\abs}{\lvert}{\rvert}
\DeclarePairedDelimiter{\set}{\lbrace}{\rbrace}

\DeclarePairedDelimiter{\parens}{\lparen}{\rparen}

\DeclareMathOperator{\Leb}{Leb}

\def\eps{{\varepsilon}}

\title[A characterization of bad approximability]{A characterization of bad approximability}
\author[F.~A.~Ram{\'i}rez]{Felipe A.~Ram{\'i}rez} \date{}
 \address{Wesleyan University, Middletown, CT, USA}
\email{framirez@wesleyan.edu}

\dedicatory{For Joaqu{\'i}n Alberto.}

\begin{document}

\frenchspacing

\begin{abstract}
  We show that badly approximable vectors are exactly those that
  cannot, for any inhomogeneous parameter, be inhomogeneously
  approximated at every monotone divergent rate. This implies in
  particular that Kurzweil's Theorem cannot be restricted to any
  points in the inhomogeneous part. Our results generalize to weighted
  approximations, and to higher irrationality exponents.
\end{abstract}

\maketitle




\section{Informal introduction}
\label{sec:intro}

One of the most fundamental objects of study in classical and modern
Diophantine approximation is the set of \emph{badly approximable
  vectors} in $\RR^d$, that is, the set
\begin{equation*}
  \bad \overset{\textrm{def}}{=} \set{x\in\RR^d : \inf_{n\in\NN} n\,\norm{nx}^d >0},
\end{equation*}
where $\norm{\cdot}$ denotes sup-norm distance to $\ZZ^d$. The set
$\bad$ has several well-known and important properties, the most
often-cited perhaps being that it has zero Lebesgue measure while
having full Hausdorff dimension. Research into the structure of $\bad$
and its generalizations continues to this day.

This paper is inspired by a question regarding Kurzweil's
Theorem~\cite{Kurzweil}, a classical result that has received renewed
attention in the past decade~\cite{ChaikaKurzweil, FayadMSTP,
  FuchsKim, Harraptwisted, Kim, SimmonsKurzweil,
  TsengSTPs}. Kurzweil's Theorem is stated in~\S\ref{sec:results}, but
for the moment it is worth mentioning that it gives an alternate
characterization of bad approximability; it is remarkable because it
uses inhomogeneous approximations, that is, expressions of the form
$\norm{nx + y}$ where $x,y\in\RR^d$, to define $\bad$, a set whose
standard definition uses only homogeneous approximations, meaning
expressions of the form $\norm{nx}$. The question is: \emph{Does
  Kurzweil's characterization of $\bad$ persist if we restrict
  inhomogeneous parameters $y$ to lie on a submanifold of $\RR^d$?}

The question itself fits into a long-running tradition of determining
to what extent classical results in Diophantine approximation survive
restriction to subsets of $\RR^d$. For more instances of this
tradition, see~\cite{AnBerVel, BadVelbadoncurves, Bermanifolds,
  BeresnevichBADonM, BDVplanarcurves, Ghoshhyperplanes, Kle03, KM98,
  hyperplanes, forall, Spr69, VV06}. Notice that the results in this
list often depend on non-degeneracy (a curvature condition) of the
manifold to which we are restricting, and in the degenerate cases
(say, affine subspaces) results tend to depend on the Diophantine
properties of the subspace. One may be tempted to venture a guess:
that the answer to the above question is ``yes'' if we restrict to
non-degenerate manifolds of $\RR^d$, or if we restrict to degenerate
manifolds satisfying some Diophantine-type condition.

In this note we consider the extreme degenerate case of singletons
(connected zero-dimensional manifolds). Here, the answer is: \emph{No,
  the opposite is true}. When one tries to restrict to points in the
inhomogeneous part, one finds instead a new characterization of bad
approximability which is intuitively the opposite of Kurzweil's
Theorem. In particular, the degenerate part of the above guess is
totally wrong for points.

In \cite{Harraptwisted} and \cite{TsengSTPs}, S.~Harrap and J.~Tseng
generalize Kurzweil's Theorem to \emph{weighted bad approximability}
and to \emph{bad approximability with respect to higher irrationality
  exponents}, respectively. The results in this paper have analogous
generalizations.

\section{Statement of results}
\label{sec:results}

Let $d\in\NN$. For a function $\psi:\NN\to\RR_{\geq 0}$, let 
\begin{equation*}
  W(\psi) \overset{\textrm{def}}{=} \set*{(x,y)\in\RR^d\times\RR^d : (\exists^\infty n\in\NN)\, \norm{nx+y}<\psi(n)}
\end{equation*}
denote the set of inhomogeneously $\psi$-approximable pairs in
$\RR^{2d}$. Let $\calD$ denote the set of all non-increasing
$\psi:\NN\to\RR_{\geq 0}$ such that $\sum_n\psi(n)^d$ diverges, and
let
\begin{equation*}
  \Pi \overset{\textrm{def}}{=} \bigcap_{\psi\in\calD} W(\psi). 
\end{equation*}
Denote by $\pi$ the projection $\RR^d\times\RR^d\to\RR^d$ to the first
copy of $\RR^d$. That is, $\pi(x,y) = x$. Let
\begin{equation*}
  V(\psi) \overset{\textrm{def}}{=} \set*{x\in\RR^d : (\textrm{a.e. } y\in\RR^d)(\exists^\infty n\in\NN)\, \norm{nx+y}<\psi(n)}.
\end{equation*}
Finally, let $\well$ denote
the complement of $\bad$, that is, the well approximable vectors. 

\theoremstyle{plain}
\newtheorem*{kurzweil}{Kurzweil's Theorem}
\begin{kurzweil}[1955]
  \begin{equation*}
    \bad = \bigcap_{\psi\in\calD}V(\psi).
  \end{equation*}
\end{kurzweil}

\begin{remark*}
  Note that the Doubly Metric Inhomogeneous Khintchine
  Theorem~\cite{CasselsintrotoDA} asserts that $W(\psi)$ has full
  Lebesgue measure in $\RR^d\times\RR^d$ for every $\psi\in\calD$,
  hence, $V(\psi)$ has full measure in $\RR^d$ by Fubini's Theorem.
\end{remark*}

The following question asks whether Kurzweil's Theorem can be
restricted to subsets of the inhomogeneous $\RR^d$-component. It was
originally asked by S.~Velani in the context of non-degenerate
manifolds, where it is sensible to think the answer might be 'yes.' We
phrase the question for general subsets of $\RR^d$, allowing that in
this generality it is not expected that the answer will
always be affirmative.

\begin{question*}
  Let $M\subset \RR^d$ be some subset (say, an affine subspace, or any
  manifold, or a fractal) supporting a measure $\mu_M$. Can we restrict
  Kurzweil's Theorem to $M$ in the inhomogeneous part? That is,
  denoting
\begin{equation*}
  V^M(\psi) \overset{\textrm{def}}{=} \set*{x\in\RR^d : (\mu_M\textrm{-a.e. } y\in\RR^d)(\exists^\infty n\in\NN)\, \norm{nx+y}<\psi(n)}
\end{equation*}
and
  \begin{equation*}
 \fixy(M) \overset{\textrm{def}}{=} \bigcap_{\psi\in\calD}V^M(\psi),
  \end{equation*}
  do we have $\bad = \fixy(M)$ as in Kurzweil's Theorem? And if not,
  what do we have instead?
\end{question*}

Our main finding is that in the case that $M=\set{y}$ is a single point with the
Dirac measure $\mu=\delta_y$, the answer is ``no.'' Not only does the set
$\fixy(y)$ always fail to coincide with $\bad$, it never contains any
badly approximable vectors at all. Furthermore, \emph{exclusion from
  $\fixy(y)$ for all $y\in\RR^d$ characterizes bad
  approximability}. We prove the following theorem.
\begin{theorem}\label{thm}
  \begin{equation*}
\well = \pi\parens*{\Pi} \qquad\textrm{hence}\qquad \bad = \RR^d\setminus \pi\parens*{\Pi}.
\end{equation*}
\end{theorem}

\begin{remark*}
Denoting
\begin{equation*}
\fixx(x) \overset{\textrm{def}}{=} \set*{y\in\RR^d : (x,y)\in\Pi},
\end{equation*}
  the theorem can be restated as
  \begin{equation*}
    x\in\bad \qquad\iff\qquad \fixx(x) = \emptyset.\tag{Thm.~\ref{thm}}
  \end{equation*}
  From the proof it is evident that $\fixx(x)$ is uncountable if
  $x\in\well\setminus\QQ^d$. We conjecture that $\fixx(x)$ is
  countable if $x\in\QQ^d$, as is the case when
  $d=1$. (See Theorem~\ref{thm:fixxcardinality}.)
\end{remark*}

\begin{remark*}
  Building on the previous remark, it should be noted that $\fixx(x)$
  and $\fixy(y)$ are ``vertical'' and ``horizontal'' fibers (respectively) of
  $\Pi$. Specifically, viewing $\Pi$ as a subset of
  $\RR^d\times\RR^d$, we see that $\fixx(x)$ is the projection of
  $\Pi\cap\parens*{\set{x}\times \RR^d}$ to the second $\RR^d$ factor,
  and $\fixy(y)$ is the projection of
  $\Pi\cap\parens*{\RR^d\times\set{y}}$ to the first $\RR^d$
  factor. Putting it another way, we have
  \begin{equation*}
    \Pi = \bigcup_{x\in\RR^d}\parens*{\set{x}\times\fixx(x)}= \bigcup_{y\in\RR^d}\parens*{\fixy(y)\times\set{y}}.
  \end{equation*}
  This point of view is especially apparent in~\S\ref{sec:thm2}.
\end{remark*}

Note that by the Doubly Metric Inhomogeneous Khintchine
Theorem~\cite{CasselsintrotoDA}, the set $W(\psi)$ has full Lebesgue measure
for every $\psi\in\calD$. But $\Pi$ is an uncountable
intersection, so it is not guaranteed that it will also have full
measure, or that it is even measurable. We prove the following.

\begin{theorem}\label{thm2}
  The set $\Pi$ is Lebesgue measurable and has measure $0$. In fact,
  for every $x\in\RR^d$, the set $\fixx(x)\subset\RR^d$ has Lebesgue
  measure $0$.
\end{theorem}

This theorem implies that for almost every $y\in\RR^d$, the measure of
$\fixy(y)$ is $0$. In Lemma~\ref{lem:Piy01} we show that $\fixy(y)\cap[0,1]^d$
always has measure $0$ or $1$. We ask: Do there exist
  $y\in\RR^d$ for which $\fixy(y)$ has full measure?

Finally, we extend these results to the cases of weighted
approximation (\S\ref{sec:weights}) and higher irrationality exponents
(\S\ref{sec:exponents}). These results appear as
Theorems~\ref{thm:weigthed},~\ref{thm2:weighted},~\ref{thm:exponents},
and~\ref{thm2:exponents}.

\begin{remark}\label{rem:pos}
  Before moving on to the proofs of the main results, we should draw
  attention to a subtlety that arises in the passage from homogeneous
  to inhomogeneous approximations. Namely, that $\norm{nx+y}$ is
  generally not equal to $\norm{(-n)x+y}$. This means that in any
  discussion about inhomogeneous approximations, one must choose
  whether one wants to allow integer values for $n$ or only consider
  positive integer values. We have chosen to consider only positive
  integer values of $n$, (so, for example, our sets $W(\psi)$ are not
  invariant under the maps $(x,y) \mapsto (-x,y)$ and
  $(x,y)\mapsto (x,-y)$, though they are invariant under the map
  $(x,y)\mapsto (-x, -y)$.) This is consistent with the convention
  used in Kurzweil's work~\cite{Kurzweil}.
\end{remark}

\section{Proof of Theorem~\ref{thm}}
\label{sec:proof}

For $x,y\in\RR^d$, $\ell\in\ZZ$, and $d\in\NN$, let 
\begin{equation*}
  S_\ell (x,y) \overset{\textrm{def}}{=} \sum_{n\geq \ell}\min_{\ell\leq m\leq n}\norm{mx+y}^d.
\end{equation*}
These sums will play an essential role here, especially through the
use of Lemma~\ref{lem:SPis}. Note that, even though we have defined
$S_\ell$ for any $\ell\in\ZZ$, we will really only need to work with
$\ell\in\NN$ (see Remark~\ref{rem:pos}). Still, there is no harm in
stating some of our simple facts about $S_\ell$---namely
Lemmas~\ref{lem:ellplusone},~\ref{lem:invariance},
and~\ref{lem:xirrationaly}---for any $\ell\in\ZZ$.

The following lemma shows that membership in $\Pi$ is equivalent to
convergence of $S_\ell$.

\begin{lemma}\label{lem:SPis}
\begin{equation*}
   (x,y)\in\Pi \quad\iff\quad S_\ell(x,y) < \infty \quad (\forall \ell\in\NN).
  \end{equation*}
\end{lemma}

\begin{proof}
  Assume there exists some $\ell\in\NN$ such that $S_\ell (x,y)=\infty$, and define the function $\psi:\NN\to\RR_{\geq 0}$ to be non-increasing and to satisfy
  \begin{equation*}
    \psi(n) =
    \min_{\ell \leq m\leq n}\norm{mx+y} \quad\textrm{for all}\quad n \geq \ell.
  \end{equation*}
  Since $S_\ell(x,y) = \infty$, we have $\psi\in\calD$. But
  $(x,y)\notin W(\psi)$ because $\norm{nx+y} < \psi(n)$ cannot be
  satisfied when $n \geq \ell$, so $(x,y)\notin\Pi$.

  On the other hand, suppose $S_\ell(x,y)<\infty$ for every
  $\ell\in\NN$, and let $\psi\in\calD$.  Now, let $\ell\in\NN$ be
  given. Since $S_\ell (x,y)$ converges while $\sum_n\psi(n)^d$
  diverges, we may choose some $n\geq \ell$ such that
  \begin{equation}\label{eq:psi>min}
    \psi(n) > \min_{\ell\leq m\leq n}\norm{mx+y}.
  \end{equation}
  (There are infinitely many such $n$.) This and the
  fact that $\psi$ is non-increasing imply that there is some
  $m\geq\ell$---specifically, the $m$ realizing the minimum
  in~(\ref{eq:psi>min})---for which
  \begin{equation}\label{eq:ineq}
    \norm{mx+y}<\psi(m).
  \end{equation}
  We have just shown that given any $\ell\in\NN$, there is some
  $m\geq \ell$ satisfying~(\ref{eq:ineq}). This implies that there are
  infinitely many $m\in\NN$ such that~(\ref{eq:ineq})
  holds. Therefore, $(x,y)\in W(\psi)$.  Since $\psi\in\calD$ was
  arbitrary, we have $(x,y)\in\Pi$.
\end{proof}

The following lemma implies that if
$S_\ell(x,y)=\infty$, then $S_{\ell+1}(x,y)=\infty$.

\begin{lemma}\label{lem:ellplusone}
  For any $\ell \in\ZZ$, we have $S_\ell(x,y)\leq \norm{\ell x + y}^d + S_{\ell+1}(x,y)$.
\end{lemma}

\begin{proof}
  We have
  \begin{multline*}
    S_\ell(x,y) = \sum_{n\geq \ell}\min_{\ell\leq m\leq
        n}\norm{mx+y}^d =
    \norm{\ell x + y}^d + \sum_{n\geq \ell+1}\min_{\ell\leq m\leq n}\norm{mx+y}^d \\
    \leq \norm{\ell x + y}^d + \sum_{n\geq
      \ell+1}\min_{\ell+1\leq m\leq n}\norm{mx+y}^d
    =\norm{\ell x + y}^d + S_{\ell+1}(x,y)
  \end{multline*}
  as needed.
\end{proof}

The following lemma is a simple consequence of bad approximability. 

\begin{lemma}\label{lem:badreturn}
  Suppose $x\in\bad$. Then there exists a constant $c:=c(x)$ such
  that for any $\eps>0$,
  \begin{equation*}
    \min\set*{n \in\NN : \norm{nx}<\eps} \geq \frac{c}{\eps^d}.
  \end{equation*}
\end{lemma}

\begin{proof}
  Let
  \begin{equation*}
    c = \inf n\,\norm{nx}^d.
  \end{equation*}
  Since $x\in\bad$, we know that $c>0$. Now, let $\eps>0$, and
  suppose that $\norm{n x}<\eps$. Then
  $c \leq n\,\norm{nx}^d < n\eps^d$, which implies the lemma.
\end{proof}

We are now able to state a proof of Theorem~\ref{thm}. We will
actually prove the following alternative statement.

\begin{theorem}[Theorem~\ref{thm}]\label{thm:restatement}
\begin{equation*}
  x\in\bad \qquad\iff\qquad \fixx(x) = \emptyset. 
  \end{equation*}
\end{theorem}

\begin{proof}
  Let $x\in\bad$, and $y\in\RR^d$. Suppose $t\in\NN$ is a ``best
  inhomogeneous approximation'' in that
  \begin{equation*}
    \norm{tx+y} < \norm{t'x + y} \quad\textrm{for all}\quad t'< t.
  \end{equation*}
  (Note that other authors may define ``best inhomogeneous
  approximation'' differently, allowing for $t\in\ZZ$ and comparing
  $\abs{t'}<\abs{t}$. See \ref{rem:pos}.) At the next best
  inhomogeneous approximation $\tilde t$, we must have
  $\tilde t x + y$ lying within $2\norm{tx+y}$ of $t x + y$, in the
  sup-norm of the torus $\RR^d/\ZZ^d$. This means that
  $\norm{(\tilde t - t) x} < 2\norm{tx+y}$. So we will have
  \begin{equation}\label{eq:inhomconvergents}
    \tilde t - t \geq \frac{c}{2^d\norm{tx + y}^d},
  \end{equation}
  by Lemma~\ref{lem:badreturn}.

  Now, recall that
  \begin{equation*}
    S_1(x,y) = \sum_{n\geq 1}\min_{1\leq m\leq n}\norm{mx+y}^d = \sum_k \norm{t_k x + y}^d(t_{k+1}-t_k),
  \end{equation*}
  where $\set{t_k}$ is the sequence of best inhomogeneous
  approximations. By~(\ref{eq:inhomconvergents}), this is bounded
  below by $\sum_k c/2^d$, which diverges. By
  Lemma~\ref{lem:SPis}, $y\notin\fixx(x)$, and we are done because
  $y$ was arbitrary.

  Now let $x\in\well$. If $x\in\QQ^d$, then it is easy to see that
  $\fixx(x)\neq\emptyset$, so assume otherwise. Since $x\in\well$,
  there is a sequence $\set{n_k}_{k\geq 0}$ of natural numbers with
  \begin{equation}
\sum_k n_k\,\norm{n_k x}^d < \infty \label{eq:conv}
\end{equation}
and
\begin{equation}
  \label{eq:geometric}
  \sum_{k\geq K} \norm{n_k x} \ll \norm{n_K x} \quad (\forall K\in\NN).
\end{equation}
Define 
\begin{equation*}
  -y = \sum_{k\geq 0} (n_k x - a_k),
\end{equation*}
where $a_k\in\ZZ^d$ is such that
$\norm*{n_k x - a_k}_\infty = \norm{n_k x}$, and let
\begin{equation*}
  N_K = \sum_{k=0}^{K-1} n_k.
\end{equation*}
Notice then that 
\begin{equation}\label{eq:Nk}
  \norm{N_K x + y} = \norm*{\sum_{k\geq K}(a_k - n_kx)}\leq \sum_{k\geq K}\norm{n_kx}.
\end{equation}
Now, let $\ell\in\NN$, and let $K_\ell\in\NN$ be such that
$N_{K_\ell} \geq \ell$. We will show that
$S_{N_{K_\ell}}(x,y)<\infty$, and this will imply, by
Lemma~\ref{lem:ellplusone}, that $S_\ell (x,y)<\infty$. Note that
\begin{multline*}
  S_{N_{K_\ell}}(x,y) \leq \sum_{K\geq K_\ell} \norm{N_K x + y}^d(N_{K+1} - N_K) \\
\overset{(\ref{eq:Nk})}{\leq} \sum_{K\geq K_\ell}\parens*{\sum_{k\geq K}\norm{n_kx}}^d n_K \overset{(\ref{eq:geometric})}{\ll} \sum_{K\geq K_\ell} n_K\,\norm{n_K x}^d,
\end{multline*}
which converges, by~(\ref{eq:conv}). So, by
Lemma~\ref{lem:ellplusone}, $S_\ell(x,y)<\infty$. Since $\ell$ was
arbitrary, Lemma~\ref{lem:SPis} tells us that $y\in\fixx(x)$, hence
$\fixx(x)\neq\emptyset$.
\end{proof}

From this proof, it is clear that if $x\in\well\setminus\QQ^d$, then
$\fixx(x)$ is uncountable. What if $x\in\QQ^d$?
Theorem~\ref{thm:fixxcardinality} in the Appendix tells us that in the
case $d=1$, we have that $\fixx(x)$ is countable whenever
$x\in\QQ$. In fact in this case we have that
\begin{equation*}\tag{Prop.~\ref{prop:fixrationals}}
  \fixx(x) = \set*{y\in\QQ : (\ZZ x + y)\cap \ZZ \neq \emptyset}.
\end{equation*}
The proof uses continued fractions, which are not available in higher
dimensions. Still, we expect an analogous statement to be true in
general.

Many questions remain. Here is another: Is $\fixy(y)$ always non-empty? 

\section{Proof of Theorem~\ref{thm2}}
\label{sec:thm2}

\begin{lemma}\label{lem:restrictint}
  We may write
  \begin{equation*}
    \Pi = \bigcap_{\psi\in\recips} W(\psi)
  \end{equation*}
  where $\recips \subset \calD$ consists of the functions in $\calD$
  that only take values that are reciprocals of natural
  numbers. 
\end{lemma}

\begin{proof} 
  To any $\psi\in \calD$ we associate the sequence
  $\bk :=\set{k_n}\in\NN^\NN$ defined by
  \begin{equation*}
    k_n-1 < \frac{1}{\psi(n)} \leq k_n,
  \end{equation*}
  and we define $\psi_\bk$ by
  \begin{equation*}
    \psi_\bk(n) = \frac{1}{k_n}.
  \end{equation*}
  That is, for each $n$ we replace $\psi(n)$ with the largest rational
  number $1/k$ ($k\in\NN$) such that $1/k \leq \psi(n)$. Clearly, we
  have $\psi_\bk(n) \leq \psi(n)$ for every $n\in\NN$, and therefore
  $W(\psi_\bk)\subseteq W(\psi)$.

  We only have to show that $\psi_\bk\in\calD$. But it is obvious that
  $\psi_\bk$ is non-increasing, so what we must show is that
  $\sum_n \psi_\bk(n)^d$ diverges. In the first case, if $k_n$ does not
  diverge to infinity, it is therefore bounded, and so $\psi_\bk$ is
  bounded below by some positive number, hence the series diverges. In
  the second case, if $k_n\uparrow\infty$, then we have $k_n>1$ for
  every sufficiently large $n$. For these $n$, we have
  $1/k_n \leq \psi(n) < 1/(k_n-1)$.  Divergence of $\sum_n\psi(n)^d$
  implies the divergence of $\sum_n1/(k_n-1)^d$, which implies
  divergence of $\sum_n 1/k_n^d = \sum_n\psi_\bk(n)^d$, so
  $\psi_\bk\in\calD$ as claimed. This proves the lemma.
\end{proof}

In view of Lemma~\ref{lem:restrictint},
we write
\begin{equation*}
  \Pi = \bigcap_{\bk\in \DD}W(\bk)
\end{equation*}
where $\DD\subset\NN^\NN$ denotes the non-decreasing sequences whose
reciprocals to $d$th powers form a divergent series, and
$W(\bk):=W(\psi_\bk)$.

In our proof that $\Pi, \fixy(y), \fixx(x)$ are measurable, we will
use some terminology from descriptive set theory. Suppose $X$ is the
set $[0,1]^d$ or $[0,1]^{2d}$. A subset of $X$ is called \emph{analytic} if
it is the projection of a Borel set in $\NN^\NN\times X$. A subset of
$X$ is called \emph{coanalytic} if it is the complement of an analytic
set. We will use the following basic facts: 
\begin{itemize}
\item Coanalytic sets are measurable. 
\item The projection of an analytic set in $\NN^\NN\times X$ is an
  analytic set in $X$.
\item Borel sets are exactly those which are both analytic and
  coanalytic.
\end{itemize}
These and more are found in~\cite{Kechrisbook,Moschovakisbook}.

\begin{lemma}\label{lem:measurable}
  The sets $\Pi, \fixy(y), \fixx(x)$ are coanalytic, hence measurable. 
\end{lemma}

\begin{proof}
  We will prove that $\Pi\comp$ is analytic. Notice that $\Pi\comp$ is
  the projection from $\NN^\NN\times [0,1]^{2d}$ to $[0,1]^{2d}$ of
  the set
  \begin{equation*}
    \WW\comp = \set*{(\bk, x,y) : \bk\in\DD, (x,y) \notin W(\bk)}.
  \end{equation*}
  Since analytic sets are closed under projections, it is enough to
  show that $\WW\comp$ is analytic. In fact, we will see that it is
  Borel.

  Notice that
\begin{equation*}
  \WW\comp = \parens*{\DD\times [0,1]^{2d}} \cap \set*{(\bk, x, y) : (x,y)\notin W(\bk)}. 
\end{equation*}
The set $\DD\times[0,1]^{2d}$ is Borel. And the set
$\set*{(\bk, x, y) : (x,y)\notin W (\bk)}$ can be expressed as
\begin{multline*}
 \bigcup_{m=1}^\infty \bigcap_{n=m}^\infty \set*{(\bk, x, y): \norm{nx+y} \geq \frac{1}{\bk(n)}}\\
= \bigcup_{m=1}^\infty \bigcap_{n=m}^\infty \bigcup_{\ell = 1}^\infty \set*{(\bk, x, y): \bk(n) = \ell\textrm{ and } \norm{nx+y} \geq \frac{1}{\ell}}.
\end{multline*}
The set of all $(\bk,x,y)\in\NN^\NN\times[0,1]^{2d}$ such that
$\bk(n)= \ell$ is Borel, and so is the set of all
$(\bk, x,y)\in\NN^\NN\times [0,1]^{2d}$ such that
$\norm{nx + y}\geq 1/\ell$. Since Borel sets are closed under
countable unions and intersections, we are done.

The same argument works for $\fixy(y)$ and $\fixx(x)$. 
\end{proof}

Now that we know the sets of interest to us are measurable, we can
proceed with calculating their measures. 

\begin{theorem}[Theorem~\ref{thm2}]\label{lem:Pix01}
  For every $x\in\RR^d$, the set $\fixx(x)$ has Lebesgue measure $0$.
\end{theorem}

\begin{proof}
  Let $x\in\RR^d$. Either the set $\set{nx\,(\bmod\,1)}_{n\geq 0}$ is
  dense in $[0,1]^d$, or its closure is a proper subtorus
  $T\subset[0,1]^d$. In the latter case, we see that if $y\notin T$,
  then $\norm{nx+y}$ is uniformly bounded below, and therefore, if
  $\psi\in\calD$ has the property that $\psi(n)\to 0$ as $n\to\infty$,
  then $(x,y)\notin W(\psi)$. Hence $(x,y)\notin\Pi$. This shows that
  $\fixx(x) \subset T$, so has Lebesgue measure $0$.

  Now suppose that $\set{nx\,(\bmod\,1)}_{n\geq 0}$ is dense in
  $[0,1]^d$. In this case, the translation $y\mapsto y+x\,(\bmod\,1)$
  is an ergodic transformation of $[0,1]^d$. Notice that if $\norm{(n+1)x +y} < \psi(n+1)$, then we have 
  \begin{equation*}
    \norm{nx + (y+x)} < \psi(n+1) \leq \psi(n).
  \end{equation*}
  This shows that $\fixx(x)+x \subseteq \fixx(x)$. That is,
  $\fixx(x)\cap[0,1]^d$ is an invariant measurable (by
  Lemma~\ref{lem:measurable}) set for the toral translation by
  $x$. Therefore, $\fixx(x)\cap[0,1]^d$ has measure $0$ or $1$.

  Now, recall that Kurzweil's Theorem asserts that 
\begin{equation*}
  \bigcap_{\psi\in\calD} \set*{x : W_x(\psi) \textrm{ has full measure}} = \bad,
\end{equation*}
where
\begin{equation*}
  W_x(\psi) \overset{\textrm{def}}{=} \set*{y \in \RR^d : (\exists^\infty n\in\NN)\, \norm{nx+y}<\psi(n)}.
\end{equation*}
This implies that for any $x\notin\bad$, there is some $\psi\in\calD$
such that 
\begin{equation*}
\Leb\parens*{W_x(\psi)\cap[0,1]^d}<1,
\end{equation*}
which in turn implies that
\begin{equation*}
\Leb\parens*{\fixx(x)\cap[0,1]^d}<1
\end{equation*}
because
\begin{equation*}
  \fixx(x) = \bigcap_{\psi\in\calD} W_x(\psi).
\end{equation*}
The previous paragraph now implies that $\fixx(x)\cap[0,1]^d$ has
measure $0$, hence $\fixx(x)$ does.

On the other hand, if $x\in\bad$, then Theorem~\ref{thm:restatement} tells us
that $\fixx(x)=\emptyset$, so it has measure $0$.
\end{proof}

It follows that $\Pi$ has measure $0$, too, and that for almost every
$y\in\RR^d$, the set $\fixy(y)$ has measure $0$. We ask: Can
  $\fixy(y)$ ever have a measure other than $0$? We show below that
the only other possible measure for $\fixy(y)\cap[0,1]^d$ is $1$.

The following lemma states that $\Pi$ is invariant under integer
contractions in the ``homogeneous'' direction, and integer dilations in
the ``inhomogeneous'' direction. We will use it to prove a $0$-$1$ law
for the measure of $\fixy(y)$.

\begin{lemma}\label{lem:diagtimesPi}
  We have
  \begin{equation*}
    \begin{pmatrix}
      (1/v)\cdot \bone_d & 0 \\ 0 & u\cdot\bone_d
    \end{pmatrix}\Pi\subseteq\Pi
  \end{equation*}
  for any $s\in(0,1]$ and $u,v\in\NN$.
\end{lemma}

\begin{proof}
  We will prove this fiber-wise. That is, we show that
  $u\cdot\fixx(x) \subseteq \fixx(x)$ and
  $(1/v)\cdot\fixy(y) \subset\fixy(y)$.

  Suppose $y\in\fixx(x)$ and let $\psi\in\calD$ and
  $u\in\NN$. Define $\widetilde\psi(n) = \frac{1}{u}\psi(un)$. Notice
  that $\widetilde \psi$ is non-increasing, and since $\set{un}_{n\geq 1}$ is
  an arithmetic sequence and $\psi$ is non-increasing, we therefore
  have that $\sum \widetilde\psi(n)^d$ diverges. Therefore
  $\widetilde\psi\in\calD$, and there are infinitely many solutions
  $n\in\NN$ to $\norm{nx+y}<\widetilde\psi(n)$. But
  $\norm{nx+uy} \leq u\, \norm{(n/u)x+y}$, and therefore there are
  infinitely many $n\in\NN$ (all of them being multiples of $u$) satisfying
  $\norm{nx+uy} < u\,\widetilde\psi(n/u) = \psi(n)$. This shows that
  $uy\in\fixx(x)$ and therefore
  \begin{equation*}
    u\cdot\fixx(x)\subseteq \fixx(x)
  \end{equation*}
  for all $u\in\NN$.

  Suppose $x\in\fixy(y)$ and let $\psi\in\calD$ and $v\in\NN$. Define
  $\widetilde\psi(n) = \psi(vn)$. It is clear that
  $\widetilde\psi\in\calD$, therefore there are infinitely many
  solutions $n\in\NN$ to $\norm{nx+y}<\widetilde\psi(n)$. Then there
  are infinitely many solutions $n\in\NN$ to
  $\norm{n(x/v)+y} < \widetilde\psi(n/v) = \psi(n)$. Since
  $\psi\in\calD$ was arbitrary, $x/v\in\fixy(y)$. We have just shown
  that
  \begin{equation*}
    \frac{1}{v}\cdot\fixy(y) \subseteq \fixy(y)
  \end{equation*}
  for all $s\in\NN$.
\end{proof}

Now a standard argument using the Lebesgue Density Theorem will prove
the desired $0$-$1$ law.

\begin{lemma}\label{lem:Piy01}
  For any $y\in\RR^d$, the set $\fixy(y)\cap[0,1]^d$ has measure $0$ or $1$.
\end{lemma}

\begin{proof}
  By Lemma~\ref{lem:diagtimesPi},
  \begin{equation*}
    \frac{1}{q}\cdot\fixy(y) \subseteq \fixy(y)
  \end{equation*}
  for all $q\in\NN$. But $\fixy(y)$ is $1$-periodic (in all $d$
  coordinate directions), and therefore $(1/q)\cdot\fixy(y)$ is
  $(1/q)$-periodic. It follows that
  \begin{equation*}
    \frac{1}{q}\cdot\fixy(y)\cap [0,1]^d\quad\textrm{and}\quad \fixy(y)\cap [0,1]^d
  \end{equation*}
  have the same Lebesgue measure, and since one is contained in the
  other, we have that
  \begin{equation}\label{eq:symmdiff}
    \Leb\parens*{\parens*{\frac{1}{q}\cdot\fixy(y)}\triangle \fixy(y)}=0,
  \end{equation}
  that is, the symmetric difference is null.

  Now, suppose that $\fixy(y)$ has positive Lebesgue measure and let
  $x\in\fixy(y)$ be a density point, meaning in particular that
  \begin{equation}\label{eq:density1}
    \Leb\parens*{x + \parens*{-\frac{1}{2q}, \frac{1}{2q}}^d\cap \fixy(y)} \sim \frac{1}{q^d} \quad (q\to\infty).
  \end{equation}
  But by~(\ref{eq:symmdiff}),
  \begin{multline*}
    \Leb\parens*{\parens*{x + \parens*{-\frac{1}{2q}, \frac{1}{2q}}^d} \cap \fixy(y)} \\
= \Leb\parens*{\parens*{x + \parens*{-\frac{1}{2q}, \frac{1}{2q}}^d}\cap \frac{1}{q}\cdot\fixy(y)},
  \end{multline*}
  and again by $1$-periodicity of $\fixy(y)$, the measure of the
  intersection of $(1/q)\cdot\fixy(y)$ with any $d$-cube of side-length
  $1/q$ is equal to $1/q^d$ times the measure of the intersection of
  $\fixy(y)$ with $[0,1]^d$, so we may conclude that
  \begin{equation}\label{eq:density2}
    \Leb\parens*{\parens*{x + \parens*{-\frac{1}{2q},\frac{1}{2q}}^d}\cap \fixy(y)} = \frac{1}{q^d}\cdot \Leb\parens*{[0,1]^d\cap \fixy(y)}.
  \end{equation}
  Finally,~(\ref{eq:density1}) and~(\ref{eq:density2}) together imply
  that $[0,1]^d\cap \fixy(y)$ has Lebesgue measure $1$, hence $\fixy(y)$
  is full, by periodicity.
\end{proof}

Our question above becomes: Do there exist $y\in\RR^d$ for which $\fixy(y)$ has
  full measure?

  Given that the sets $\Pi$ and $\fixx(x)$ have zero Lebesgue measure,
  it is natural to ask: What are their Hausdorff dimensions? Since
  $\pi(\Pi)$ has full measure, it therefore has Hausdorff dimension
  $d$, hence we have that the Hausdorff dimension of $\Pi$ is at least
  $d$. What is it exactly? In the $d=1$ case, one can be explicit in
  the proof of Theorem~\ref{thm} using continued fractions, and find
  $x\in\RR$ such that $\fixx(x)$ contains Cantor-type sets whose
  dimensions we can calculate. Perhaps this kind of strategy can be
  pushed, and even adapted for general $d$.

\section{Weighted approximation}
\label{sec:weights}

Let $\simplex \overset{\textrm{def}}{=} \set*{r\in\RR^d : r_i \geq 0, \sum_i r_i = 1}$ be the standard simplex. For
$r=(r_i)\in\simplex$ and $x=(x_i)\in\RR^d$, denote
\begin{equation*}
  \norm{x}_r \overset{\textrm{def}}{=} \parens*{\max_{1\leq i\leq d}\norm{x_i}^{1/r_i}}^{1/d},
\end{equation*}
and notice that $\norm{\cdot}_r$ coincides with $\norm{\cdot}$ (in $\RR^d$) when
$r = (1/d, \dots, 1/d)$.  Define
\begin{equation*}
  \bad(r) \overset{\textrm{def}}{=} \set*{x\in\RR^d :  \inf n\,\norm{nx}_r^d > 0},
\end{equation*}
the set of $r$-weighted badly approximable vectors. For $\psi:\NN\to\RR_{\geq 0}$ and $r\in\simplex$, let
\begin{equation*}
  W(r,\psi) \overset{\textrm{def}}{=} \set*{x\in\RR : (\exists^\infty n\in\NN)\, \norm{nx}_r < \psi(n)}.
\end{equation*}
Let
\begin{equation*}
  \Pi_r \overset{\textrm{def}}{=} \bigcap_{\psi\in\calD} W(r,\psi). 
\end{equation*}
Finally, let
\begin{equation*}
  V(r, \psi) \overset{\textrm{def}}{=} \set*{x\in\RR^d : (\textrm{a.e. } y\in\RR^d)(\exists^\infty n\in\NN)\, \norm{nx+y}_r<\psi(n)}.
\end{equation*}
Kurzweil's Theorem has the following generalization to the weighted
case, due to S.~Harrap.
\begin{theorem}[\cite{Harraptwisted}]
  \begin{equation*}
    (\forall r\in\simplex)\quad\bad(r) = \bigcap_{\Psi\in\calD}V(r, \Psi).
  \end{equation*}
\end{theorem}

We have the following weighted version of Theorem~\ref{thm}.

\begin{theorem}\label{thm:weigthed}
  \begin{equation*}
(\forall r\in\simplex)\quad\well(r) = \pi\parens*{\Pi_r} \qquad\textrm{hence}\qquad \bad(r) = \RR^d\setminus \pi\parens*{\Pi_r}.
\end{equation*}
\end{theorem}

\begin{proof}
  All statements in~\S\ref{sec:proof} remain true if all instances of
  \begin{equation*}
  \norm{\cdot}, W(\psi), \bad, \well, \fixx(x)
\end{equation*}
are replaced with
\begin{equation*}
\norm{\cdot}_r, W(r, \psi), \bad(r), \well(r), \fixx_r(x),
\end{equation*}
their weighted analogues.
\end{proof}

The arguments from~\S\ref{sec:thm2} also hold for weighted
approximations, giving us the following.

\begin{theorem}\label{thm2:weighted}
  For any $r\in\simplex$ the set $\Pi_r$ is Lebesgue measurable and
  has measure $0$. In fact, for every $x\in\RR^d$, the set
  $\fixx_r(x)\subset\RR^d$ has Lebesgue measure $0$.
\end{theorem}

\section{Higher irrationality exponents}
\label{sec:exponents}

For $\sigma\geq 1/d$, consider the set
\begin{equation*}
  \sigma\bad \overset{\textrm{def}}{=} \set*{x\in\RR : \liminf_{n\to\infty}n^\sigma\,\norm{nx} > 0}
\end{equation*}
of $\sigma$-\emph{badly approximable vectors}. Naturally, 
\begin{equation*}
  \sigma\well \overset{\textrm{def}}{=}\RR\setminus\sigma\bad
\end{equation*}
is referred to as the set of $\sigma$-\emph{well approximable
  vectors}. Notice that $\sigma\bad$ and $\sigma\well$ coincide with
the usual badly approximable and well approximable vectors when
$\sigma=1/d$.

Kurzweil's Theorem has been generalized to this setting by J.~Tseng. For $t>0$, let
\begin{equation*}
  \calD^t = \set*{\psi^t : \psi \in\calD}.
\end{equation*}
Tseng proves the following.

\begin{theorem}[\cite{TsengSTPs}]\label{thm:tseng}
  \begin{equation*}
    (\forall \sigma \geq 1/d)\quad\sigma\bad = \bigcap_{\psi\in\calD^{1/\sigma}} V(\psi).
  \end{equation*}
\end{theorem}

Theorems~\ref{thm} and~\ref{thm2} also transfer easily to higher
irrationality exponents. Let
\begin{equation*}
  \sigma\Pi = \bigcap_{\psi\in\calD^\sigma} W(\psi).
\end{equation*}
We have

\begin{theorem}\label{thm:exponents}
  \begin{equation*}
    (\forall \sigma \geq 1/d)\quad\sigma\well = \pi\parens*{\sigma\Pi} \qquad\textrm{hence}\qquad \sigma\bad = \RR^d\setminus \pi\parens*{\sigma\Pi}.
  \end{equation*}
\end{theorem}

\begin{remark*}
  Notice that in Theorem~\ref{thm:exponents} the intersection is over
  $\psi\in\calD^\sigma$, while in Theorem~\ref{thm:tseng} the
  intersection is over $\psi\in\calD^{1/\sigma}$. 
\end{remark*}

\begin{proof}
  As in the proof of Theorem~\ref{thm:weigthed}, we only have to
  adjust the proof of Theorem~\ref{thm}. In this case, the main
  adjustment is to re-define $S_\ell(x,y)$ as
\begin{equation*}
  S_\ell (x,y) \overset{\textrm{def}}{=} \sum_{n\geq \ell}\min_{\ell\leq m\leq n}\norm{mx+y}^{1/\sigma}.
\end{equation*}
(Notice that if $\sigma = 1/d$, then we have it as originally
defined.) Then the proof of Lemma~\ref{lem:SPis} is easily adapted to
prove a version saying that membership in $\sigma\Pi$ is equivalent to
convergence of $S_\ell$, as before. Lemmas~\ref{lem:ellplusone}
and~\ref{lem:badreturn} both remain true if every instance of ``$d$'' is
replaced with ``$1/\sigma$.'' And finally the logic of the proof of
Theorem~\ref{thm} goes through with the new inputs.
\end{proof}

Theorem~\ref{thm2} generalizes to higher irrationality exponents even
more easily, because $\calD^\sigma \supset \calD$, and therefore
$\sigma\Pi\subset \Pi$.

\begin{theorem}\label{thm2:exponents}
  For every $\sigma\geq 1/d$, the set $\sigma\Pi$ is Lebesgue
  measurable and has measure $0$. In fact, all of its ``vertical''
  fibers have $d$-dimensional Lebesgue measure $0$.
\end{theorem}

In this case, it would be more informative to study the Hausdorff
measures of $\sigma\Pi$ and its horizontal and vertical fibers.

\appendix

\subsection*{Acknowledgments}
\label{acknowledgments}

I thank Victor Beresnevich, Martin Goldstern, Cameron Hill, and Sanju
Velani for helpful conversations and correspondence. I would also like
to acknowledge the support my wife, Yeni, who was pregnant with our
second child when I started this project; without her, Luna and
Joaqu{\'i}n would certainly never let me pursue any projects at all.

\section{The case $d=1$}
\label{sec:d=1}

There are some things that are easier to prove when $d=1$, but
which may still be true in the general case. Chief among them is the
following refinement of Theorem~\ref{thm}.

\begin{theorem}\label{thm:fixxcardinality}
  \begin{equation*}
    \#\fixx(x) = 
    \begin{cases}
      \aleph_0 &\textrm{if}\quad x\in\QQ \\
      2^{\aleph_0} &\textrm{if}\quad x\in\well\setminus\QQ \\
      0&\textrm{if}\quad x\in\bad.
    \end{cases}
  \end{equation*}
\end{theorem}

Theorem~\ref{thm} already tells us that $\fixx(x)=\emptyset$ when
$x\in\bad$, and its proof shows us that $\fixx(x)$ is uncountable
whenever $x\in\well\setminus\QQ$. It is only left to show that
$\fixx(x)$ is countable when $x\in\QQ$. This follows immediately from
the following proposition, which completely describes $\fixy(y)$ and
$\fixx(x)$ when $x,y\in\QQ$.

\begin{proposition}\label{prop:fixrationals}~
 \begin{enumerate}
  \item If $y\in\QQ$, then
    \begin{equation*}
      \fixy(y) = \set*{x\in\QQ : (\ZZ x + y)\cap \ZZ \neq \emptyset}.
    \end{equation*}
  \item If $x\in\QQ$, then
    \begin{equation*}
      \fixx(x) = \set*{y\in\QQ : (\ZZ x + y)\cap \ZZ \neq \emptyset}.
    \end{equation*}
  \end{enumerate}
\end{proposition}

In the proof, we will use the next two lemmas about $S_\ell$. 

\begin{lemma}\label{lem:invariance}
  $S_{\ell+k}(x,y) = S_\ell(x, kx + y)$.
\end{lemma}

\begin{proof}
  We have
  \begin{align*}
    S_{\ell+k}(x,y) &= \sum_{n\geq \ell + k}\parens*{\min_{\ell + k \leq m \leq n}\norm*{mx+y}}\\
                    &= \sum_{n\geq \ell + k}\parens*{\min_{\ell + k \leq m \leq n}\norm*{(m-k)x+ (kx +y)}}\\
    \intertext{and letting $\tilde n = n-k$ and $\tilde m= m-k$,}
                    &= \sum_{\tilde n\geq \ell}\parens*{\min_{\ell \leq \tilde m \leq \tilde n}\norm*{\tilde m x+ (kx +y)}}\\
                    &= S_\ell(x,kx+y),
  \end{align*}
  as claimed.
\end{proof}

\begin{lemma}\label{lem:xirrationaly}
  If $x \in\RR\setminus\QQ$ and $y\in\RR$ and $k\in\ZZ$ is such that
  $kx+y\in\ZZ$, then
  \begin{equation*}
    S_\ell(x,y) =
    \begin{cases}
      0 &\text{if } \ell \leq k \\
      \infty &\text{if } \ell > k.
    \end{cases}
  \end{equation*}
  In particular,
  \begin{equation*}
    S_\ell(x,0) =
    \begin{cases}
      0 &\text{if } \ell \leq 0 \\
      \infty &\text{if } \ell > 0
    \end{cases}
  \end{equation*}
  for any irrational $x$.
\end{lemma}

\begin{proof}
  The lemma reduces to the $y=0$ case after observing that $S_\ell$ is
  invariant by integer shifts in both coordinates, and applying
  Lemmas~\ref{lem:ellplusone} and~\ref{lem:invariance}.

  Let us therefore prove the ``in particular.'' Let $\ell \leq
  0$. Then
  \begin{equation*}
    \min_{\ell\leq m\leq n}\norm{mx} = 0
  \end{equation*}
  for all $n\geq \ell$, and so $S_\ell(x,0)=0$.

  Now note that
  \begin{equation*}
    S_1(x,0) = \sum_{n \geq 1}\parens*{\min_{1 \leq m \leq n}\norm{mx}} = \sum_{n\geq 0}\norm{q_nx}\parens{q_{n+1} - q_n}
  \end{equation*}
  where $q_n$ denotes a convergent of the continued fraction
  $x=[a_0; a_1, a_2, \dots]$. This is bounded by
  \begin{align}
    \sum_{n \geq 0}\norm{q_nx}\parens{q_{n+1} - q_n}    
    &\geq \sum_{n \geq 0}\frac{q_{n+1} - q_n}{q_{n+1}  + q_n} \nonumber\\
    &\geq \sum_{n \geq 0}\parens*{1 - \frac{2q_n}{q_{n+1}  + q_n}} \label{eq:marker} \\
    &> \sum_{n \geq 0}\parens*{1 - \frac{2}{(a_{n+1}+1)}}, \nonumber
  \end{align}
  and if there are infinitely many partial quotients greater than $1$,
  then there will be infinitely many terms bounded below by $1/3$, and
  the sum will diverge.

  It is only left to treat the case where
  $x = [a_0; a_1, \dots, a_M, \overline{1}]$. Here we can
  bound~(\ref{eq:marker}) by
  \begin{align*}
    \sum_{n\geq 0}\parens*{1 - \frac{2q_n}{q_{n+1}  + q_n}} &\geq \sum_{n\geq M}\parens*{1 - \frac{2q_n}{q_{n+1}  + q_n}} \\
                                                            &= \sum_{n\geq M}\parens*{1 - \frac{2q_n}{q_{n+2}}} \\
                                                            &= \sum_{n\geq M}\parens*{1 - 2\parens*{\frac{q_n}{q_{n+1}}}\parens*{\frac{q_{n+1}}{q_{n+2}}}}.
  \end{align*}
  Now, since
  $\lim_{n\to\infty}\parens*{\frac{q_n}{q_{n+1}}}=\varphi^{-1}$, where
  $\varphi$ is the golden ratio, we have that
  \begin{equation*}
    \parens*{1 - 2\parens*{\frac{q_n}{q_{n+1}}}\parens*{\frac{q_{n+1}}{q_{n+2}}}}\to 1-\frac{2}{\varphi^2},
  \end{equation*}
  which is positive, hence the sum diverges.

  Now, by Lemma~\ref{lem:ellplusone}, we have that
  $S_\ell (x,0)=\infty$ for all $\ell \geq 1$.
\end{proof}

The following special case of Proposition~\ref{prop:fixrationals}
tells us that the intersections of all sets coming from the divergence
case of Khintchine's Theorem is exactly $\QQ$. This is an interesting
fact in itself. 

\begin{lemma}\label{lem:fixy0}
  Let $d=1$. Then $\fixy(0)=\QQ$. 
\end{lemma}

\begin{proof}
  Notice that if $x\in\QQ$, then $(x,0)\in W(\psi)$ for any
  $\psi\in\calD$, because we have $\norm{nx}=0$ for infinitely many
  $n$. Therefore, $\fixy(0)\supseteq\QQ$. On the other hand, if
  $x\in\RR\setminus\QQ$, then $S_1(x,0)=\infty$, by
  Lemma~\ref{lem:xirrationaly}. So $x\notin\fixy(0)$ by
  Lemma~\ref{lem:SPis}. Therefore $\fixy(0)=\QQ$ as claimed.
\end{proof}

\begin{proof}[Proof of Proposition~\ref{prop:fixrationals}]
  If $x,y$ are rational and $\ZZ x + y$ contains no integers, then the
  set $\set{nx+y\,(\bmod 1)}_{n\in\ZZ}$ is finite, hence the
  expression $\norm{nx+y}$ is bounded below uniformly over
  $n\in\ZZ$. Therefore $S_\ell(x,y)=\infty$ for all $\ell\in\NN$, and
  by Lemma~\ref{lem:SPis}, $x\notin\fixy(y)$ and $y\notin\fixx(x)$. (In
  fact, this argument also shows that $\fixx(x)$ contains no
  irrationals if $x$ is rational.) On the other hand if $\ZZ x + y$ contains integers,
  then $\norm{nx+y}=0$ on an arithmetic sequence of $n\in\NN$, so
  $S_\ell(x,y)<\infty$ for all $\ell\in\NN$, and by
  Lemma~\ref{lem:SPis}, $x\in\fixy(y)$ and $y\in\fixx(x)$.

  It is only left to prove that if $y\in\QQ$, then $\fixy(y)$ contains
  no irrationals. Write $y=p/q$ for some $p\in\ZZ$ and
  $q\in\NN$. Lemma~\ref{lem:diagtimesPi} implies that
  $\fixy(y)\subseteq \fixy(qy)=\fixy(p)$. And, by periodicity, we have
  $\fixy(p)=\fixy(0)$. But Lemma~\ref{lem:fixy0} states that
  $\fixy(0)=\QQ$, hence $\fixy(y)\subseteq\QQ$. 
\end{proof}

The following corollary is immediate.

\begin{corollary}\label{cor:bothratorirr}
  If $(x,y)\in\Pi$, then either $x$ and $y$ are both rational, or they
  are both irrational.
\end{corollary}

\begin{lemma}\label{lem:yirrational}
  If $1, x,y$ are rationally dependent and $x$ (or $y$) is irrational,
  then $(x,y)\notin\Pi$.
\end{lemma}

\begin{proof}
  Suppose $k_1x + k_2 y + k_3 = 0$, with
  $(k_1,k_2,k_3)\in\ZZ^3\setminus\set{(0,0,0)}$. Notice that if either
  $x$ or $y$ is irrational, then they both are. Then, by
  Lemma~\ref{lem:xirrationaly}, $S_\ell(x,k_2y)=\infty$ for all
  sufficiently large $\ell$ and therefore $k_2y\notin\fixx(x)$, by
  Lemma~\ref{lem:SPis}. But Lemma~\ref{lem:diagtimesPi} tells us that
  $k_2\cdot\fixx(x)\subseteq \fixx(x)$, therefore
  $y\notin\fixx(x)$. The lemma follows.
\end{proof}

Lemma~\ref{lem:yirrational} raises the question of whether $\Pi$
includes all $(x,y)$ such that $1,x,y$ are rationally independent. The
answer is ``no.'' For any $x$, the set of $y$ such that $1, x, y$ are
rationally independent has full measure. But by Theorem~\ref{thm2}
the set $\fixx(x)$ has measure $0$.


\bibliographystyle{plain}

\bibliography{../bibliography}

\begin{thebibliography}{10}

\bibitem{AnBerVel}
Jinpeng An, Victor Beresnevich, and Sanju Velani.
\newblock Badly approximable points on planar curves and winning.
\newblock preprint, 2014.

\bibitem{BadVelbadoncurves}
Dzmitry Badziahin and Sanju Velani.
\newblock Badly approximable points on planar curves and a problem of
  {D}avenport.
\newblock {\em Math. Ann.}, 359(3-4):969--1023, 2014.

\bibitem{Bermanifolds}
Victor Beresnevich.
\newblock Rational points near manifolds and metric {D}iophantine
  approximation.
\newblock {\em Ann. of Math. (2)}, 175(1):187--235, 2012.

\bibitem{BeresnevichBADonM}
Victor Beresnevich.
\newblock Badly approximable points on manifolds.
\newblock {\em Invent. Math.}, 202(3):1199--1240, 2015.

\bibitem{BDVplanarcurves}
Victor Beresnevich, Detta Dickinson, and Sanju Velani.
\newblock Diophantine approximation on planar curves and the distribution of
  rational points.
\newblock {\em Ann. of Math. (2)}, 166(2):367--426, 2007.
\newblock With an Appendix II by R. C. Vaughan.

\bibitem{CasselsintrotoDA}
J.~W.~S. Cassels.
\newblock {\em An introduction to {D}iophantine approximation}.
\newblock Cambridge Tracts in Mathematics and Mathematical Physics, No. 45.
  Cambridge University Press, New York, 1957.

\bibitem{ChaikaKurzweil}
Jon Chaika.
\newblock Shrinking targets for {IET}s: extending a theorem of {K}urzweil.
\newblock {\em Geom. Funct. Anal.}, 21(5):1020--1042, 2011.

\bibitem{FayadMSTP}
Bassam Fayad.
\newblock Mixing in the absence of the shrinking target property.
\newblock {\em Bull. London Math. Soc.}, 38(5):829--838, 2006.

\bibitem{FuchsKim}
Michael Fuchs and Dong~Han Kim.
\newblock On {K}urzweil's 0-1 law in inhomogeneous {D}iophantine approximation.
\newblock {\em Acta Arith.}, 173(1):41--57, 2016.

\bibitem{Ghoshhyperplanes}
Anish Ghosh.
\newblock A {K}hintchine-type theorem for hyperplanes.
\newblock {\em J. London Math. Soc. (2)}, 72(2):293--304, 2005.

\bibitem{Harraptwisted}
Stephen Harrap.
\newblock Twisted inhomogeneous {D}iophantine approximation and badly
  approximable sets.
\newblock {\em Acta Arith.}, 151(1):55--82, 2012.

\bibitem{Kechrisbook}
Alexander~S. Kechris.
\newblock {\em Classical descriptive set theory}, volume 156 of {\em Graduate
  Texts in Mathematics}.
\newblock Springer-Verlag, New York, 1995.

\bibitem{Kim}
Dong~Han Kim.
\newblock The shrinking target property of irrational rotations.
\newblock {\em Nonlinearity}, 20(7):1637--1643, 2007.

\bibitem{Kle03}
D.~Kleinbock.
\newblock Extremal subspaces and their submanifolds.
\newblock {\em Geom. Funct. Anal.}, 13(2):437--466, 2003.

\bibitem{KM98}
D.~Y. Kleinbock and G.~A. Margulis.
\newblock Flows on homogeneous spaces and {D}iophantine approximation on
  manifolds.
\newblock {\em Ann. of Math. (2)}, 148(1):339--360, 1998.

\bibitem{Kurzweil}
J.~Kurzweil.
\newblock On the metric theory of inhomogeneous diophantine approximations.
\newblock {\em Studia Math.}, 15:84--112, 1955.

\bibitem{Moschovakisbook}
Yiannis~N. Moschovakis.
\newblock {\em Descriptive set theory}, volume 155 of {\em Mathematical Surveys
  and Monographs}.
\newblock American Mathematical Society, Providence, RI, second edition, 2009.

\bibitem{hyperplanes}
F.~A. {Ram{\'{i}}rez}.
\newblock Khintchine types of translated coordinate hyperplanes.
\newblock {\em Acta Arith.}, 170(3):243--273, 2015.

\bibitem{forall}
F.~A. {Ram{\'{i}}rez}, D.~{Simmons}, and F.~{S{\"u}ess}.
\newblock Rational approximation of affine coordinate subspaces of {E}uclidean
  space.
\newblock {\em Acta Arith.}, 177(1):91--100, 2017.

\bibitem{SimmonsKurzweil}
David Simmons.
\newblock An analogue of a theorem of {K}urzweil.
\newblock {\em Nonlinearity}, 28(5):1401--1408, 2015.

\bibitem{Spr69}
V.~G. Sprind{\v{z}}uk.
\newblock {\em Mahler's problem in metric number theory}.
\newblock Translated from the Russian by B. Volkmann. Translations of
  Mathematical Monographs, Vol. 25. American Mathematical Society, Providence,
  R.I., 1969.

\bibitem{TsengSTPs}
Jimmy Tseng.
\newblock On circle rotations and the shrinking target properties.
\newblock {\em Discrete Contin. Dyn. Syst.}, 20(4):1111--1122, 2008.

\bibitem{VV06}
R~Vaughan and Sanju Velani.
\newblock Diophantine approximation on planar curves: the convergence theory.
\newblock {\em Inventiones mathematicae}, 166:103--124, 2006.

\end{thebibliography}

\end{document}